\documentclass[10pt]{article}
\usepackage{latexsym}
\usepackage{graphicx}
\usepackage{amssymb}
\usepackage{algorithm}
\usepackage{algorithmic}
\usepackage{comment}
\usepackage[affil-it]{authblk}
\usepackage{amsmath,amsthm}
\usepackage{amsfonts}
\usepackage{booktabs}
\usepackage[labelsep=period]{caption}

\theoremstyle{plain}
\newtheorem{thm}{Theorem}[section]
\newtheorem{theorem}[thm]{Theorem}

\newtheorem{corollary}[thm]{Corollary}
\newtheorem{example}[thm]{Example}
\newtheorem{case}[thm]{Special Case}

\theoremstyle{definition}

\usepackage[sorting=none,giveninits=true]{biblatex} % to get references how they want them
\DeclareNameAlias{sortname}{first-last} % sort first name, last name
\renewbibmacro{in:}{} % remove "In:"
\DeclareFieldFormat{journaltitle}{\mkbibemph{#1}\isdot} % dots in journal names are not full stops

\DeclareFieldFormat[article,inbook,incollection,inproceedings,patent,thesis,unpublished]{citetitle}{#1\isdot}
\DeclareFieldFormat[article,inbook,incollection,inproceedings,patent,thesis,unpublished]{title}{#1\isdot}
\DeclareFieldFormat[article,inbook,incollection,inproceedings,patent,thesis,unpublished]{date}{#1\isdot} % add parentheses around date

\makeatletter
\def\@maketitle{%
  \newpage
  \null
  \vskip 2em%
  \begin{center}%
  \let \footnote \thanks
    {\Large\bfseries \@title \par}%
    \vskip 1.5em%
    {\normalsize
      \lineskip .5em%
      \begin{tabular}[t]{c}%
        \@author
      \end{tabular}\par}%
    \vskip 1em%
    {\normalsize \@date}%
  \end{center}%
  \par
  \vskip 1.5em}
\makeatother

\begin{document}

\title{Novel Theorems and Algorithms Relating to the Collatz Conjecture}
\author{Michael R. Schwob}
\affil{Department of Statistics and Data Sciences\\
                University of Texas at Austin\\
                Austin, TX, 78712
                \\ schwob@utexas.edu}
%thanks{schwob@utexas.edu}
\author{Peter Shiue}
\affil{Department of Mathematical Sciences\\
                University of Nevada, Las Vegas\\
                Las Vegas, NV, 89154 \\ shiue@unlv.nevada.edu}
%\thanks{shiue@unlv.nevada.edu}
\author{Rama Venkat}
\affil{Howard R. Hughes College of Engineering\\
                University of Nevada, Las Vegas\\
                Las Vegas, NV, 89154 \\ rama.venkat@unlv.edu}
%\thanks{rama.venkat@unlv.edu}

\date{}

\maketitle

% Keywords command
\providecommand{\keywordsh}[1]
{
  \small	
  \textbf{\textit{Keywords---}} #1
}

% Keywords command
\providecommand{\subjectclassification}[1]
{
  \small	
  \textbf{\textit{2010 Mathematics Subject Classification---}} #1
}
	
\begin{abstract}
Proposed in 1937, the Collatz conjecture has remained in the spotlight for mathematicians and computer scientists alike due to its simple proposal, yet intractable proof. In this paper, we propose several novel theorems, corollaries, and algorithms that explore relationships and properties between the natural numbers, their peak values, and the conjecture. These contributions primarily analyze the number of Collatz iterations it takes for a given integer to reach 1 or a number less than itself, or the relationship between a starting number and its peak value.
\end{abstract}

\keywordsh{Collatz conjecture, 3x+1 problem, geometric series, iterations, algorithms, peak values}

\subjectclassification{11B75,11B83,68R01,68R05}
	
\section{Introduction}

In 1937, the German mathematician Lothar Collatz proposed a conjecture that states the following sequence will always reach 1, regardless of the given positive integer $n$: if the previous term in the sequence is even, the next term is half of the previous term; if the previous term in the sequence is odd, the next term is 3 times the previous term plus 1 \cite{andrei1998}. This conjecture is considered impossible to prove given modern mathematics. In fact, Paul Erdős has claimed that ``mathematics may not be ready for such problems" \cite{Erdos}. Jeffrey Lagarias has echoed this sentiment and provides a thorough summary of numerous results concerning the conjecture in \cite{Lagarias}.

Despite today's supposed inability to prove the Collatz conjecture, several papers have outlined discoveries related towards the seemingly impossible proof. Perhaps the most notable recent development was made by Terence Tao, who showed that most orbits of the Collatz map attain almost bounded values \cite{tao2019}. An excellent review of his paper was published in the \textit{College Mathematics Journal} \cite{hartnett2019}. This significant development illuminates one potential avenue to proving the Collatz Conjecture; however, other paths to the proof remain viable.

Some mathematicians have attempted to break the Collatz conjecture up into several smaller problems or have rephrased the conjecture in terms that may be more tractable. For example, Ren proves that if his Reduced Collatz conjecture (RCC) is proven, then that proof can be applied to the Collatz conjecture. Ren also proved that half of the natural numbers follow reduced Collatz dynamics; however, the second half of the proof for the RCC remains unsolved \cite{ren2019}. While exploring reduced Collatz dynamics may be a promising route for the proof of the Collatz conjecture, it is likely beneficial to propose other potential routes to the proof \cite{arslan2018,ren2019b,zarnowski2019}.

Several manuscripts have outlined various, unsuccessful paths that mathematicians have taken to solve the conjecture \cite{lagarias1985,latourette2007,van2005}. Perhaps the proof of this mesmerizing conjecture is so desired because the conjecture itself finds use in various applications. The Collatz conjecture is used in high-uncertainty audio signal encryption \cite{renza2019}, image encryption \cite{ballesteros2018}, dynamic software watermarking \cite{ma2019}, and information discovery \cite{idowu2015}. Although the lack of a proof does not preclude applications of the Collatz conjecture, mathematical phenomenon related to the conjecture may be of interest in other applications, such as cryptography or information theory.

In section 2, we present novel theorems and corollaries that explore a mapping that follows the Collatz conjecture; this section is broken up into two subsections with the first focusing on determining the number of Collatz iterations for a natural number to reach 1 and the second subsection focusing on an analysis of peak values. In section 3, an algorithm to obtain a directed graph of the conjecture is provided, as well as an algorithm that analyzes iterations of the conjecture and an algorithm that determines peak values in Collatz sequences. Lastly, we summarize these novel theorems, corollaries, and algorithms in section 4, and we relate them to the Collatz conjecture.

%Lastly, an analysis is provided that examines the natural numbers that require significantly many iterations in the Collatz conjecture to reach 1 or reach a value less than itself.
	
\section{The Conjecture and Related Theorems}

Let $\mathbb{N}$ be the set of all positive integers and $n \in \mathbb{N}$. Collatz defined the following map:

\[ T(n)= \begin{cases} 
\frac{n}{2} & \text{, if }n\text{ is even} \\
3n+1 & \text{, if }n\text{ is odd}
\end{cases}
\]

A sequence can be formed by iterating over $T$: $n$, $T(n)$, $T^2(n)=T(T(n)),...,T^k(n)=T(T^{k-1}(n))$ for $k\in\mathbb{N}$. The Collatz conjecture asserts that the sequence obtained by iterating $T$ always reaches the integer 1, no matter which positive integer $n$ begins the sequence. A diagram that depicts the relationship between positive integers is shown in Fig. 1. We start our exploration of the Collatz conjecture with the following theorem.

\begin{theorem}
    Let $a_{n+1}=(3b+1)a_n+b$ and $n \ge 2$, where $a_1$ and $b$ are given positive, odd integers. Then,
	
	$$\text{(i): } T(a_{n+1})=(3b+1)T(a_n)$$
	
	$$\text{(ii): } T(a_{n+1})=(3b+1)^nT(a_1)$$
\end{theorem}

\begin{proof}
	
	For (i): By the Collatz conjecture, we have 
	$$T(a_{n+1})=3(a_{n+1})+1=3\left[(3b+1)a_n+b\right]+1$$
	
	$$=9ba_n+3a_n+3b+1=(3b+1)(3a_n+1)=(3b+1)T(a_n)$$
	
	For (ii): We prove through a combination of (i) and mathematical induction. Let $n=1$. Then, 
	
	$$T(a_{n+1})=(3b+1)^nT(a_1)=(3b+1)T(a_1),$$
	which is true by (i). Now, assume that $T(a_{n+1})=(3b+1)^nT(a_1)$. If we show that $T(a_{(n+1)+1})=(3b+1)^{(n+1)}T(a_1)$, then we prove (ii). By (i),
	
	$$T(a_{(n+1)+1})=T(a_{n+2})=(3b+1)T(a_{n+1}),$$
	and by our assumption,
	$$(3b+1)T(a_{n+1})=(3b+1)(3b+1)^nT(a_1)=(3b+1)^{n+1}T(a_1).$$
	Therefore, we have proved through mathematical induction that $T(a_{n+1})=(3b+1)^nT(a_1)$.
	
\end{proof}

\begin{figure}[h]
	\caption{Collatz Conjecture Visualized}
	\vspace{0.5cm}
	\centering
	\includegraphics[scale=0.3]{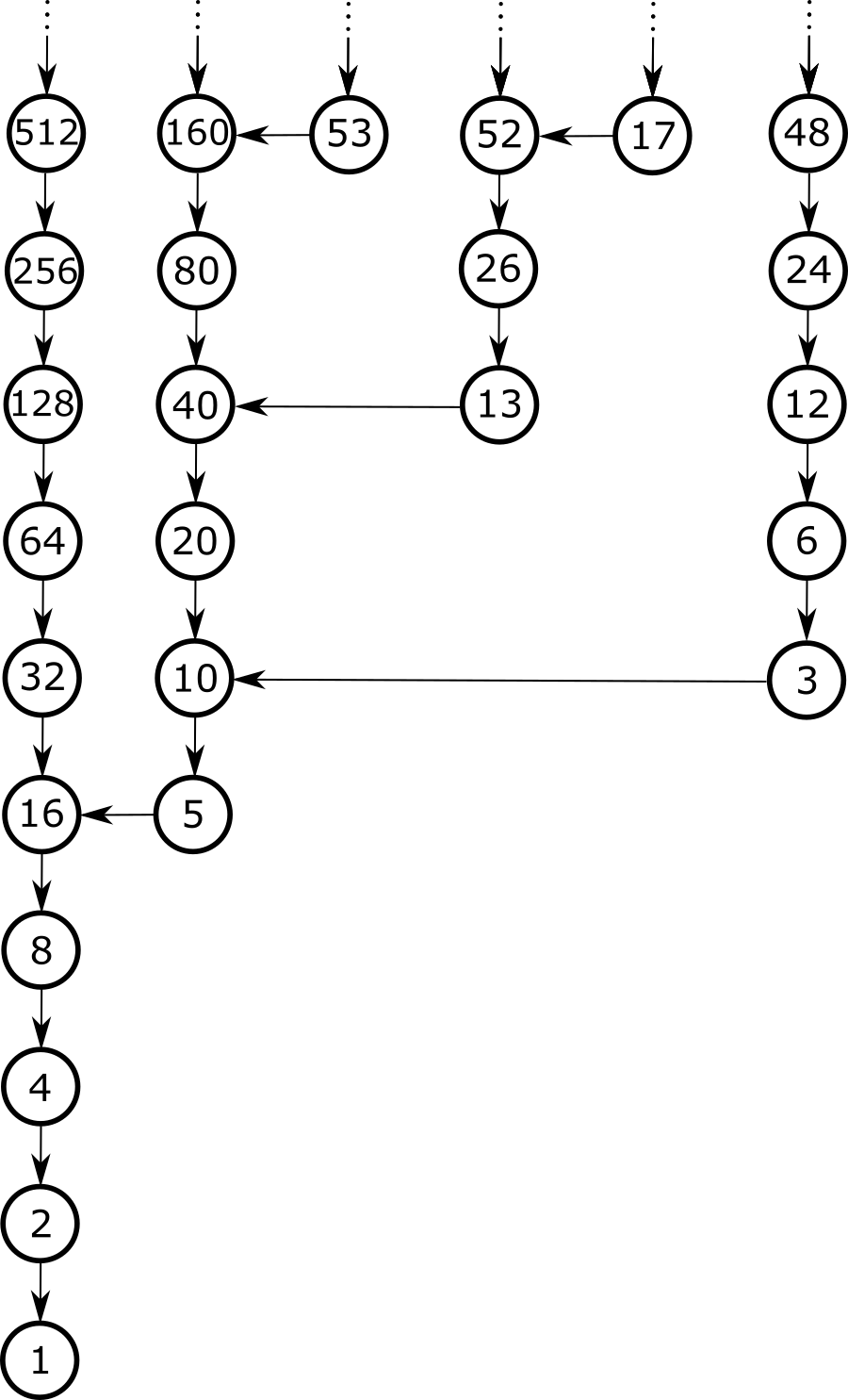}
\end{figure}

From this theorem, we can establish a unique relationship for a particular choice in $b$. This relationship is delineated in our first corollary.

\begin{corollary}

	Let $a_{n+1}=4^ka_n+\frac{4^{k}-1}{3}$. Then, 
	
	$$\text{(i): } T(a_{n+1})=4^kT(a_n),$$
	$$\text{(ii): } T(a_{n+1})=(4^k)^nT(a_1),$$
	which follows from Theorem 2.1 when $3b+1=4^k$.
	
\end{corollary}

We can then use (i) from corollary 2.2 to help establish the following:

\begin{corollary}
	Let $k=1$, $b=1$, and $a_n=2j+1$, where $j \in \mathbb{N}$. Since $a_{n+1}=4a_n+1=8j+5$, we have $T(8j+5)=4T(2j+1)$.
\end{corollary}

Suppose $j=1$ and examine $T(13)$. By corollary 2.3, $T(13)=4T(3)=2^2T(2j+1)$. We can extract a convenient and meaningful relationship from this corollary. Note that the sequence starting with integer 13 takes 9 iterations to reach 1 in the Collatz conjecture, and the sequence that starts at integer 3 takes 7 iterations. Therefore, the sequence starting at 13 takes 2 more iterations than the sequence starting at 3. Now, take $j=2$ and examine $T(21)$. By corollary 2.3, $T(21)=2^2T(5)$. Through the Collatz conjecture, the sequence starting at integer 5 takes 5 iterations to reach 1, while the sequence starting at integer 21 takes 7 iterations to reach 1. Therefore, there is a difference of two iterations between sequences starting at $8j+5$ and $2j+1$ for any $j \in \mathbb{N}$. We see that, if $a_n=2j+1$ reaches 1 in the conjecture, then $a_{n+1}$ reaches 1.

%The Collatz conjecture's proof seems unattainable with modern mathematics because a proof that can generalize a similar relationship for all natural numbers is difficult to provide. However, this relationship is a discovery that can aid the pursuit of such a proof. 

\subsection{Counting Iterations} 

Most approaches to proving the Collatz conjecture will inevitably detail the number of iterations it takes for a given number $n$ to reach 1. Hence, we provide relationships between $\mathbb{N}$ and the number of Collatz iterations it would take to complete a sequence ending in 1; this will be of direct interest to the academic community attempting to prove the conjecture. 

In the theorem that follows, we provide a more general relationship that implies if $a_n$ reaches 1 in finitely many iterations, then $a_{n+1}=(3b+1)a_n+b$ also reaches 1 in finitely many iterations.

\begin{theorem}
	Suppose $a_n$ reaches 1 with $j$ iterations of the Collatz conjecture. Then, $a_{n+1}=4^ka_n+\frac{4^{k}-1}{3}$ reaches 1 with $2k+j$ iterations.
\end{theorem}

\begin{proof}
	Note that $T^j(a_n)=1$, where $j$ is the number of iterations required of the Collatz conjecture for $a_n$ to reach 1. From corollary 2.2, we have $T(a_{n+1})=4^kT(a_n)=2^{2k}T(a_n)$. Since $T^{2k}(2^{2k})=1$, we have $$T^{2k}(2^{2k}T(a_n))=T(a_n),$$ which gives us $$T^{2k}(T(a_{n+1}))=T^{2k}(2^{2k}T(a_n))=T(a_n).$$ 
	For any positive integer $j$, $$T^{j-1}\left[T^{2k+1}(a_{n+1})\right] = T^{j-1}\left[T(a_n)\right]=T^j(a_n).$$
	Since $T^j(a_n)=1$, we have $T^{2k+j}(a_{n+1})=1$.
\end{proof}

\begin{example}
    Fix $k=1$, and suppose we let $a_1=3$. If we take $a_{n+1}=4a_n+1$, then $a_2=13$, $a_3=53$, and $a_4=213$. By theorem 2.4, $a_2$, $a_3$, and $a_4$ will reach 1 in a finite number of iterations so long as $a_1$ does. From Fig. 1, we know that $T^7(3)=1$. Therefore, 13 reaches 1 with $7+2=9$ iterations. Similarly, 53 and 213 reach 1 with 11 and 13 iterations, respectively. This is depicted in Fig. 2. Another example, where $a_1=5$, is depicted in Fig. 3.
\end{example}

\begin{figure}[h]
	\caption{Collatz Conjecture with $a_1=3$}
	\vspace{0.5cm}
	\centering
	\includegraphics[scale=0.3]{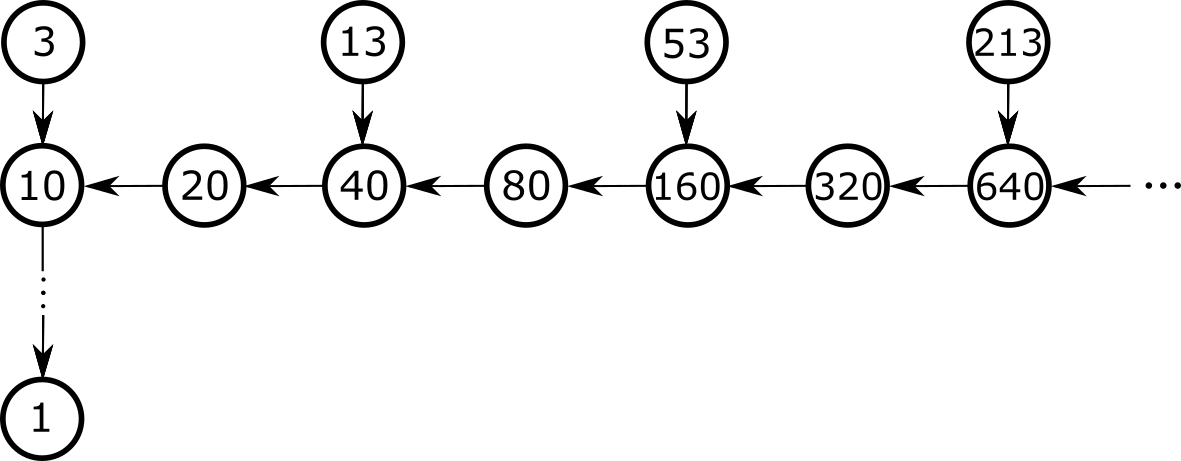}
\end{figure}

\begin{figure}[h]
	\caption{Collatz Conjecture with $a_1=5$}
	\vspace{0.5cm}
	\centering
	\includegraphics[scale=0.3]{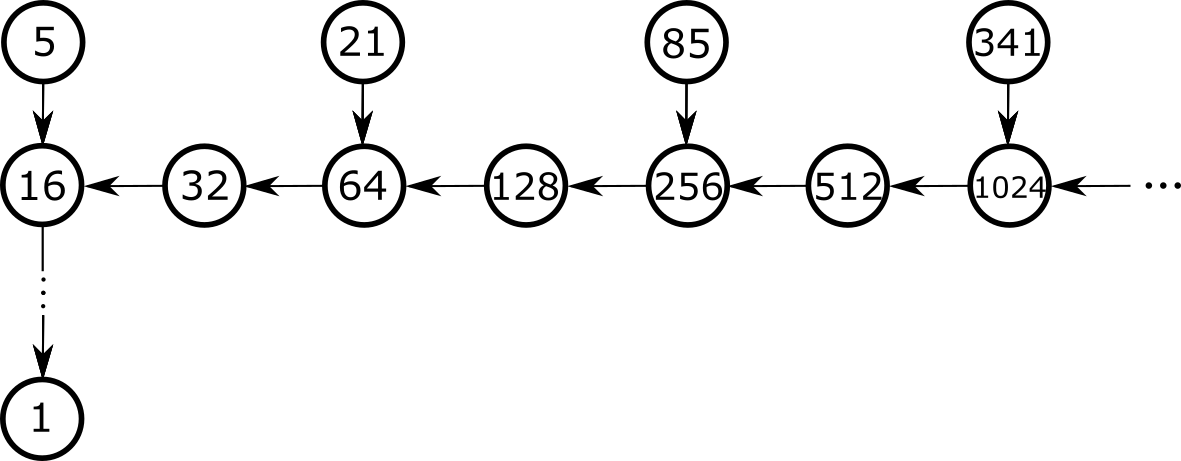}
\end{figure}

A special case arises from Theorem 2.4 when $a_n=1$ and, consequently, $a_k=a_{n+1}=4^k+\frac{4^{k}-1}{3}=\frac{4^{k+1}-1}{3}$. By Theorem 2.4, $a_k$ reaches 1 with $2k+3$ Collatz iterations. Note that $a_k$ takes on the form of a geometric series with ratio $4$: $$a_k=1+4+4^2+\dots+4^k.$$ Therefore, this particular geometric series shares a unique relationship with the Collatz conjecture. Perhaps this discovery may inspire new approaches to the Collatz conjecture's proof that utilize series and their properties.

%Let $s_n=1+4+4^2+\dots+4^n=\frac{4^{n+1}-1}{3}=Q_1$, which is always odd. Therefore, $$Q_2=T(Q_1)=3\left( \frac{4^{n+1}-1}{3}\right) +1=4^{n+1}=2^{2n+2}.$$ Since $Q_2$ is even, we obtain $Q_3=T(Q_2)=T^2(Q_1)=2^{2n+1}$, which requires an additional $2n+1$ iterations to reach 1 (i.e., $T^{2n+1}(Q_3)=1$). Therefore, $T^{2n+3}(Q_1)=1$.

While general relationships would be the most helpful to establish a proper proof of the Collatz conjecture, an analysis of specific relationships may prove nontrivial. In the following theorem, we explore such relationships similar to the special sequences relating to the Collatz conjecture in \cite{jenber2018}.

\begin{theorem}
	Let $k \in \mathbb{N}$ and $t\in \mathbb{N} \cup \{0\}$. We have the following relationships:
	
	(i) If $n=4k+1$, then the number of iterations until the Collatz sequence reaches a number less than $n$ is 3. That is, $T^3(4k+1)<4k+1$.
	
	(ii) If $n=16t+3$, then $T^6(16t+3) < 16t+3$.
	
	(iii) If $n=32t+11$ or $n=32t+23$, then $T^8(n)<n$.
	
	(iv) If $n=128t+7$, $n=128t+15$, or $n=128t+59$, then $T^{11}(n)<n$.
	
	(v) If $n=256k+39$, $n=256k+79$, $n=256k+95$, $n=256k+123$, $n=256k+199$, or $n=256k+219$ then $T^{13}(n)<n$.
\end{theorem}

\begin{proof}
	
	This proof can be obtained analytically or computationally. Since all points in this proof follow the same structure, we only analytically prove (i) and (ii).
	
	For (i): We have that $T(4k+1)=3(4k+1)+1=12k+4$. Similarly, we can obtain $T^2(4k+1)=T(12k+4) =\frac{12k+4}{2}= 6k+2$, and $T^3(4k+1)=T(6k+2)=3k+1<4k+1$.
	
	For (ii): We continue to follow the Collatz sequence. We have $T(16t+3)=48t+10$, $T^2(16t+3)=T(48t+10)=24t+5$, $T^3(16t+3)=T(24t+5)=72t+16$, $T^4(16t+3)=T(72t+16)=36t+8$, $T^5(16t+3)=T(36t+8)=18t+4$, and $T^6(16t+3)=T(18t+4)=9t+2<16t+3$.
	
	The proof for the remaining portions of this theorem are proved through a similar process, which can be supplemented with computation. Examples of computations for the last three portions of this theorem are provided in Tables 1-3.
\end{proof}

% \usepackage{booktabs}
\begin{table}[]
	\centering
	\caption{Theorem 2.6 (iii)}
	\begin{tabular}{@{}cccc@{}}
		\toprule
		\textbf{Iteration} & \textbf{t=1: n=43} & \textbf{t=7: n=247} & \textbf{t=14: n=459} \\ \midrule
		\textbf{1}         & 130                & 742                 & 1378                 \\
		\textbf{2}         & 65                 & 371                 & 689                  \\
		\textbf{3}         & 196                & 1114                & 2068                 \\
		\textbf{4}         & 98                 & 557                 & 1034                 \\
		\textbf{5}         & 49                 & 1672                & 517                  \\
		\textbf{6}         & 148                & 836                 & 1552                 \\
		\textbf{7}         & 74                 & 418                 & 776                  \\
		\textbf{8}         & 37                 & 209                 & 388                  \\ \bottomrule
	\end{tabular}
\end{table}

% \usepackage{booktabs}
\begin{table}[]
	\centering
	\caption{Theorem 2.6 (iv)}
	\begin{tabular}{@{}cccc@{}}
		\toprule
		\textbf{Iteration} & \textbf{t=1: n=135} & \textbf{t=7: n=911} & \textbf{t=14: n=1851} \\ \midrule
		\textbf{1}         & 406                 & 2734                & 5554                  \\
		\textbf{2}         & 203                 & 1367                & 2777                  \\
		\textbf{3}         & 610                 & 4102                & 8332                  \\
		\textbf{4}         & 305                 & 2051                & 4166                  \\
		\textbf{5}         & 916                 & 6154                & 2083                  \\
		\textbf{6}         & 458                 & 3077                & 6250                  \\
		\textbf{7}         & 229                 & 9232                & 3125                  \\
		\textbf{8}         & 688                 & 4616                & 9376                  \\
		\textbf{9}         & 344                 & 2308                & 4688                  \\
		\textbf{10}        & 172                 & 1154                & 2344                  \\
		\textbf{11}        & 86                  & 577                 & 1172                  \\ \bottomrule
	\end{tabular}
\end{table}

% \usepackage{booktabs}
\begin{table}[]
	\centering
	\caption{Theorem 2.6 (v)}
	\begin{tabular}{@{}cccc@{}}
		\toprule
		\textbf{Iteration} & \textbf{t=1: n=295} & \textbf{t=7: n=1871} & \textbf{t=14: n=3707} \\ \midrule
		\textbf{1}         & 886                 & 5614                 & 11122                 \\
		\textbf{2}         & 443                 & 2807                 & 5561                  \\
		\textbf{3}         & 1330                & 8422                 & 16684                 \\
		\textbf{4}         & 665                 & 4211                 & 8342                  \\
		\textbf{5}         & 1996                & 12634                & 4171                  \\
		\textbf{6}         & 998                 & 6317                 & 12514                 \\
		\textbf{7}         & 499                 & 18952                & 6257                  \\
		\textbf{8}         & 1498                & 9476                 & 18772                 \\
		\textbf{9}         & 749                 & 4738                 & 9386                  \\
		\textbf{10}        & 2248                & 2369                 & 4693                  \\
		\textbf{11}        & 1124                & 7108                 & 14080                 \\
		\textbf{12}        & 562                 & 3554                 & 7040                  \\
		\textbf{13}        & 281                 & 1777                 & 3520                  \\ \bottomrule
	\end{tabular}
\end{table}

The consequence of this theorem is that each of the positive integers $n$ that can be written as any of the above forms will eventually become a number less than itself. If $n$ does not abide by the above forms, its pattern is unknown. If a similar, general relationship were to be discovered such that each positive integer $n$ becomes a number less than itself, one can then prove the Collatz conjecture. 

The next three examples reveal the number of iterations it takes for an integer to reach a number less than itself. The first example arises from a special case of Theorem 2.6(i), and the last two examples pertain to two particular geometric series.

\begin{example}
    Consider $3^m$, where $m$ is even. Then, we have $$3^m \equiv (-1)^m \equiv 1 (\text{mod } 4).$$ We can then write $3^m=4k+1$ for some $k \in \mathbb{N}$. By Theorem 2.6(i), $T^3(3^m)<3^m$. Similarly, $7^m=4k_1+1$ and $11^m=4k_2+1$ for some $k_1,k_2 \in \mathbb{N}$. In general, we can write $n=(3+4t)^m=4k+1$ for some $k \in \mathbb{N}$ and even $m$. Then, by Theorem 2.6, $T^3\left[(3+4t)^m\right]<(3+4t)^m$.
\end{example}

\begin{example}
    Let $s_m=1+4^n+4^{2n}+\cdots+4^{mn}$. Then, $s_m$ will reduce to a number less than itself in 3 Collatz iterations. We can show this by rewriting $$s_m=1+4(4^{n-1}+\cdots+4^{mn-1})=1+4t,$$ for some $t \in \mathbb{N}$. Therefore, by Theorem 2.6(i), $s_m$ reduces to a number less than itself in 3 Collatz iterations. That is, $T^3(s_m)<s_m$.
\end{example}
%\begin{theorem}
%Any integer that arises from the following geometric series with ratio 4 will reduce to a number less than itself in 3 Collatz iterations:

%\end{theorem}

%\begin{proof}

%	Let $s_m=1+4^n+4^{2n}+\cdots+4^{mn}=\frac{(4^n)^{m+1}-1}{4^n-1}=Q_1$, which will be an odd integer. Then, $Q_2=T(Q_1)=\frac{3(4^n)^{m+1}+4^n-4}{4^n-1}=\frac{4[3(4)^{n(m+1)-1}+4^{n-1}-1]}{4^n-1}$. Since $Q_2$ is even, $Q_4=T(Q_3)=T(T(Q_2))=\frac{3(4)^{n(m+1)-1}+4^{n-1}-1}{4^n-1}$. Note that we can rewrite $Q_4$ as
%	$$\frac{3(4)^{n(m+1)-1}+4^{n-1}-1}{4^n-1}=\frac{4(4)^{n(m+1)-1}-4^{n(m+1)-1}+4^{n-1}-1}{4^n-1}$$
%	$$=\frac{4^{n(m+1)}-4^{n(m+1)-1}+4^{n-1}-1}{4^n-1}.$$
%	For all $m,n\ge 1$, we have $4^{n-1} < 4^{n(m+1)-1}$. So,
%	$$Q_4 < \frac{4^{n(m+1)}-4^{n(m+1)-1}+4^{n(m+1)-1}-1}{4^n-1}=\frac{4^{n(m+1)}-1}{4^n-1}=Q_1.$$
%	Therefore, $T^3(s_m)<s_m$.
%\end{proof}

\begin{example}
    Let $s_m=1+(an+b)+(an+b)^2+\cdots + (an+b)^m$, where $\text{gcd}(a,b)=4$. Then, $s_m$ reduces to a natural number less than itself in 3 Collatz iterations. We can rewrite $s_m$ as $$s_m=1+(an+b)\left[1+(an+b)+\cdots+(an+b)^{m-1}\right]$$
    $$=1+4\left(\frac{a}{4}n+\frac{b}{4}\right)\left[1+(an+b)+\cdots+(an+b)^{m-1}\right]$$
    $$=1+4t,$$
    for some $t \in \mathbb{N}$. By Theorem 2.6(i), $T^3(s_m)<s_m$.
\end{example}

%\begin{theorem}
	%Any integer that arise from the following geometric series with ratio $12n+4$ reduces to a natural number less than itself in 3 Collatz iterations:
	
%\end{theorem}

%\begin{proof}

	%Let $s_m=1+(12n+4)+(12n+4)^2+\cdots + (12n+4)^m$.  So, %$s_m=\frac{(12n+4)^{m+1}-1}{12n+4-1}=\frac{(12n+4)^{m+1}-1}{3(4n+1)}=Q_1$. Since $Q_1$ is odd, we can obtain $Q_2=\frac{(12n+4)^{m+1}-1+4n+1}{4n+1}=\frac{4[(3n+1)(12n+4)^m+n]}{4n+1}.$	Here, $Q_2$ is even, so we can obtain $Q_3=\frac{2[(3n+1)(12n+4)^m+n]}{4n+1}$ and $Q_4=\frac{(3n+1)(12n+4)^m+n}{4n+1}$. We want to show that 
	%$$\frac{(3n+1)(12n+4)^m+n}{4n+1} < \frac{(12n+4)^{m+1}-1}{3(4n+1)}$$
	%$$\implies 3(3n+1)(12n+4)^m+3n+1<(12n+4)^{m+1}$$
	%$$\implies 3(12n+4)^m+1<\frac{(12n+4)^{m+1}}{3n+1}=\frac{(3n+1)^{m+1}4^{m+1}}{3n+1}$$
	%$$\implies 3(12n+4)^m+1<4(12n+4)^m,$$
	%which is clearly true $\FORall n\in \mathbb{N}$. Therefore, $T^3(s_m)<s_m$.
%\end{proof}

It is important to note that these last two examples can be proven by computation without using Theorem 2.6(i).
\begin{proof}
    For the first example, let $s_m=1+4^n+4^{2n}+\cdots+4^{mn}=\frac{(4^n)^{m+1}-1}{4^n-1}=Q_1$, which will be an odd integer. Then, $Q_2=T(Q_1)=\frac{3(4^n)^{m+1}+4^n-4}{4^n-1}=\frac{4[3(4)^{n(m+1)-1}+4^{n-1}-1]}{4^n-1}$. Since $Q_2$ is even, $Q_4=T(Q_3)=T(T(Q_2))=\frac{3(4)^{n(m+1)-1}+4^{n-1}-1}{4^n-1}$. Note that we can rewrite $Q_4$ as
$$\frac{3(4)^{n(m+1)-1}+4^{n-1}-1}{4^n-1}=\frac{4(4)^{n(m+1)-1}-4^{n(m+1)-1}+4^{n-1}-1}{4^n-1}$$
$$=\frac{4^{n(m+1)}-4^{n(m+1)-1}+4^{n-1}-1}{4^n-1}.$$
For all $m,n\ge 1$, we have $4^{n-1} < 4^{n(m+1)-1}$. So,
$$Q_4 < \frac{4^{n(m+1)}-4^{n(m+1)-1}+4^{n(m+1)-1}-1}{4^n-1}=\frac{4^{n(m+1)}-1}{4^n-1}=Q_1.$$
Therefore, $T^3(s_m)<s_m$.
\end{proof}

\begin{proof}
    For the second example, let $s_m=1+(12n+4)+(12n+4)^2+\cdots + (12n+4)^m$, where $a=12$ and $b=4$.  So, $s_m=\frac{(12n+4)^{m+1}-1}{12n+4-1}=\frac{(12n+4)^{m+1}-1}{3(4n+1)}=Q_1$. Since $Q_1$ is odd, we can obtain $Q_2=\frac{(12n+4)^{m+1}-1+4n+1}{4n+1}=\frac{4[(3n+1)(12n+4)^m+n]}{4n+1}.$	Here, $Q_2$ is even, so we can obtain $Q_3=\frac{2[(3n+1)(12n+4)^m+n]}{4n+1}$ and $Q_4=\frac{(3n+1)(12n+4)^m+n}{4n+1}$. We want to show that 
$$\frac{(3n+1)(12n+4)^m+n}{4n+1} < \frac{(12n+4)^{m+1}-1}{3(4n+1)}$$
$$\implies 3(3n+1)(12n+4)^m+3n+1<(12n+4)^{m+1}$$
$$\implies 3(12n+4)^m+1<\frac{(12n+4)^{m+1}}{3n+1}=\frac{(3n+1)^{m+1}4^{m+1}}{3n+1}$$
$$\implies 3(12n+4)^m+1<4(12n+4)^m,$$
which is clearly true $\forall n\in \mathbb{N}$. Therefore, $T^3(s_m)<s_m$.
\end{proof}

\subsection{Peak Values}

Analyzing the number of iterations required for a given $n$ to complete a Collatz sequence may lead to interesting results related to the Collatz conjecture, as can be seen above. Another fascinating set of values that may be of interest are those that are the maximum value within a Collatz sequence; we call these values ``peak values." The following seven theorems relate to peak values and their properties.

\begin{theorem}
    Let $n \in \mathbb{N}$, and let $T^*(n)=p$ denote the peak value of the Collatz sequence starting at $n$. Then, $p \equiv 0 (\text{mod } 4)$ (i.e., $4 \mid p$).
\end{theorem}

\begin{proof}
    First, we show that $p$ is even through contradiction. If $p$ is odd, then $p=2m+1$, where $m \in \mathbb{N}$. Then, $$T(p)=3p+1=6m+4>2m+1=p,$$
    which means that $p$ is not the peak value in its Collatz sequence. Thus, $p$ is even (i.e., $\frac{p}{2} \in \mathbb{N}$).
    Now, we claim that $\frac{p}{2}$ is even. If $\frac{p}{2}$ were odd, then $$T\left(\frac{p}{2}\right)=3\left(\frac{p}{2}\right)+1=\frac{3p+2}{2}>p,$$
    so $p$ is not the peak value. By contradiction, $\frac{p}{2}$ is even. Hence, $\frac{p}{4} \in \mathbb{N} \implies 4 \mid p$.
\end{proof}

\begin{theorem}
    Let $u$ be the odd number immediately before p. Then, $u \equiv 1 (\text{mod } 4)$ (i.e., $u=4l+1$ for some $l \in \mathbb{N}$).
\end{theorem}

\begin{proof}
    Since $u$ is odd, either $u=4l+1$ or $u=4l+3$ for some $l \in \mathbb{N}$. Suppose that $u=4l+3$. Then, $p=T(u)=12l+10=2(6l+5)$. Since $2 \nmid 6l+5$, then $4 \nmid p$. However, $4 \mid p$ by theorem 2.10. So through contradiction, we know that $u=4l+1$. 
\end{proof}

\begin{theorem}
    Let $v$ be the odd number immediately after $p$. Then, $v < u$.
\end{theorem}

\begin{proof}
    By Theorem 2, $u=4l+1 \implies p=T(u)=12l+4$. So, $\frac{p}{4}=3l+1$. Let $3l+1=2^ev$, where $e$ is the largest power such that $2^e \mid 3l+1$ (i.e., $v=\frac{3l+1}{2^e}$). Then, $\frac{3l+1}{2^e}<4l+1 \implies v<u$.
\end{proof}

%%%%%%%%%%%%%%%%%%%%%%%%%%%%%%%%%%%%%%%%%
\begin{comment}

\begin{theorem}
    Let $u$ be the odd number before $p$, and let $u=8l+5$ for some $l \in \mathbb{N}$. Then, $8 \mid p$ if $l$ is odd, and $16 \mid p$ if $l$ is even.
\end{theorem}

\begin{proof}
    We have $p=T(u)=24l+16=8(3l+2)$. Hence, if $l$ is odd, then $8 \mid p$, and if $l$ is even, then $16 \mid p$.
\end{proof}

Theorem 2.13 can be extended to the following theorem, which can be proven using the same technique we used for 2.13.

\begin{theorem}
    If $u$ is an odd number before $p$, then we have the following cases:
    \begin{enumerate}
        \item Let $u=8l+5$. Then, $8 \mid p$ and $v=\frac{p}{8}$.
        \item Let $u=16l+5$. Then, $16 \mid p$ and $v=\frac{p}{16}$.
        \item Let $u=32l+21$. Then, $32 \mid p$ and $v=\frac{p}{32}$.
        \item Let $u=64l+21$. Then, $64 \mid p$ and $v=\frac{p}{64}$.
    \end{enumerate}
\end{theorem}

\end{comment}
%%%%%%%%%%%%%%%%%%%%%%%%%%%%%%%%%%%%%%%%%

\begin{theorem}
    Let $u=2\cdot4^tk+\frac{4^{t+1}-1}{3}$, where $t \ge 1$. If $k$ is odd, then $2\cdot4^t \mid p$ and $v=3k+2$.
\end{theorem}

\begin{proof}
    We see that $p=T(u)=6\cdot 4^tk+4^{t+1}=4^t(6k+4)=2\cdot4^t(3k+2)$. Since $k$ is odd and $3k+2 \not\equiv 0 (\text{mod }2)$, we obtain $2\cdot4^t \mid p$ and $v=3k+2$.
\end{proof}

In what follows, we have three special cases of Theorem 2.13 and two examples that demonstrate the use of this theorem.

\begin{case}
    Let $t=1$. Then, $u=8k+5 \implies 8 \mid p$ and $v=3k+2$ so long as $k$ is odd.
\end{case}

\begin{case}
    Let $t=2$. Then, $u=32k+21 \implies 32 \mid p$ and $v=3k+2$ so long as $k$ is odd.
\end{case}

\begin{case}
    Let $t=3$. Then, $u=2\cdot4^3k+85 \implies 128 \mid p$ and $v=3k+2$ so long as $k$ is odd.
\end{case}

\begin{example}
    Let $u=13=8\cdot1+5$ and $p=T(u)=40$. Then, $8 \mid p$ and $v=3\cdot1+2=5$.
\end{example}

\begin{example}
    Let $u=53=32\cdot1+21$ and $p=160$. Then, $32 \mid p$ and $v=3\cdot1+2=5$.
\end{example}

While Theorem 2.13 explores the case when $k$ is odd, the following theorem covers the case when $k$ is even.

\begin{theorem}
    Let $u=2\cdot4^tk+\frac{4^{t+1}-1}{3}$. If $k$ is even and $k=2l$ for some $l \in \mathbb{N}$, then we have the following cases:
    \begin{enumerate}
        \item If $l$ is even, then $4^{t+1} \mid p$ and $v=3l+1$.
        \item If $l$ is odd, let $2^e \mid (3l+1)$ and $2^{e+1} \nmid (3l+1)$. Then, $2^{2t+e+2} \mid p$ and $v=\frac{3l+1}{2^e}$.
    \end{enumerate}
\end{theorem}

\begin{proof}
    We let $p = 2\cdot4^t(3k+2)=2\cdot4^t(6l+2)=4\cdot4^t(3l+1)=4^{t+1}(3l+1)$.
    \begin{enumerate}
        \item If $l$ is even, then $3l+1 \not\equiv 0 (\text{mod } 2)$. Hence, $4^{t+1} \mid p$ and $v=3l+1$.
        \item If $l$ is odd, then let $2^e \mid (3l+1)$ and $2^{e+1} \nmid (3l+1)$. We have $p=T(u)=4^{t+1}\cdot2^e\cdot \frac{3l+1}{2^e}$. Thus, $2^{2t+e+2} \mid p$ and $v=\frac{3l+1}{2^e}$.
    \end{enumerate}
\end{proof}

Just as we did for Theorem 2.13, we provide the following special cases and examples for Theorem 2.19.

\begin{case}
    Let $t=1$. Then, $u=8k+5$ and $k=2l$.
    \begin{enumerate}
        \item If $l$ is even, then $16 \mid p$ and $v=3l+1$.
        \item If $l$ is odd, let $2^e \mid (3l+1)$ and $2^{e+1} \nmid (3l+1)$. Then, $2^{4+e} \mid p$ and $v=\frac{3l+1}{2^e}$.
    \end{enumerate}
\end{case}

\begin{case}
    Let $t=2$. Then, $u=32k+21$ and $k=2l$.
    \begin{enumerate}
        \item If $l$ is even, then $64 \mid p$ and $v=3l+1$.
        \item If $l$ is odd, let $2^e \mid (3l+1)$ and $2^{e+1} \nmid (3l+1)$. Then, $2^{6+e} \mid p$ and $v=\frac{3l+1}{2^e}$.
    \end{enumerate}
\end{case}

\begin{case}
    If $\frac{3l+1}{2^e}=1$, then $3l+1=2^e \implies l=\frac{2^e-1}{3}$. Hence, $l=1$ if $e=2$, $l=5$ if $e=4, \cdots $; thus, $$u=2\cdot 4^t(2l)+\frac{4^{t+1}-1}{3}=\frac{1}{3}[4^{t+1}\cdot2^e-1]=\frac{1}{3}[4^{t+1}\cdot4^m-1]=\frac{1}{3}[4^a-1],$$ where $e=2m$ is a nonzero even integer and $a=t+1+m \in \mathbb{N}$. Hence, if $u$ is an odd number such that $u=\frac{1}{3}[4^a-1]$, then $p=T(u)=4^a$ and $v=1$. Note that $u$ will reach 1 in $2a+1$ steps.
\end{case}

\begin{example}
    Let $t=1$, $u=21=8\cdot2+5$, $k=2$ and $p=64$. Then, $2^e \mid (3l+1) \implies 2^e \mid 4$ and $2^{e+1} \nmid 4$. Thus, $e=2$ and $2^{2t+2+e}=2^6=64 \mid p$. Lastly, $v=\frac{3l+1}{2^e}=1$.
\end{example}

\begin{example}
    Let $t=1$, $u=37=8\cdot4+5$, $k=4$ and $p=112$. Then, $2^e \mid (3l+1) \implies 2^e \mid 7$ and $2^{e+1} \nmid 7$. Thus, $e=0$ and $2^{2t+2+e}=2^4=16 \mid p$. Lastly, $v=\frac{3l+1}{2^e}=7$.
\end{example}

In what follows, we present a general result with some accompanying examples for a certain form in $u$ and its respective peak values.

\begin{theorem}
    Let $u=4^t k+\frac{4^t-1}{3}$. If $k$ is even (i.e., $k=2l$ for some $l \in \mathbb{N}$) and $t \in \mathbb{N}$, then $4^t \mid p$ and $v=6l+1$.
\end{theorem}

\begin{proof}
    We have $p=T(u)=3k(4^t)+4^t=4^t(3k+1)=4^t(6l+1)$. Since $6l+1 \not\equiv 0 (\text{mod } 2)$ (i.e., $2 \nmid (6l+1)$), $4^t \mid p$ and $v=6l+1$.
\end{proof}

\begin{example}
    Consider the case when $t=1$. Then, $u=4k+1=8l+1$ and $p=4(6l+1) \implies 4 \mid p$ and $v=6l+1$. Suppose $u=65=8(8)+1$. Then, $l=8$ and $T(u)=4(6\cdot 8+1)=196$. That is, $4 \mid 196$, and $v=6l+1=49$.
\end{example}

\begin{example}
    Consider the case when $t=2$. Then, $u=4^2k+5=32l+5$ and $p=4^2(6l+1) \implies 4^2 \mid p$ and $v=6l+1$. Suppose $u=101=32(3)+5$. Then, $l=3$, $p=4^2(6l+1)=16(6\cdot3+1)=16\cdot19=304$, and $v=6l+1=19$.
\end{example}

\begin{example}
    Consider the case when $t=3$. Then, $u=4^3k+21=128l+21$ and $p=4^3(6l+1) \implies 4^3 \mid p$ and $v=6l+1$. Suppose $u=149=128(1)+21$. Then, $l=1$, $64 \mid p$, and $v=6l+1=7$.
\end{example}

The following theorem is similar to Theorem 2.24; however, it explores the cases that arise when $k$ is odd, specifically when $l$ is odd and when $l$ is even. Four accompanying examples are provided.

\begin{theorem}
    Let $u=4^tk+\frac{4^t-1}{3}$ and $k$ be odd (i.e., $k=2l+1$ for some $l \in \mathbb{N}$). 
    
    \begin{enumerate}
        \item If $l$ is odd, then $2^{2t+1} \mid p$ and $v=3l+2$.
        \item Suppose $l$ is even, and let $2^e \mid (3l+2)$ and $2^{e+1} \nmid (3l+2)$. Then, $2^{2t+e+1} \mid p$ and $v=\frac{3l+2}{2^e}$.
    \end{enumerate}
\end{theorem}

\begin{proof}

We use a similar proofing technique as we used for Theorem 2.19.

\begin{enumerate}
    \item We have that $p=T(u)=3\cdot 4^t k+4^t=3\cdot 4^t(2l+1)+4^t=4^t(6l+4)=2\cdot4^t(3l+2)=2\cdot 2^{2t}(3l+2)=2^{2t+1}(3l+2)$. Since $2 \nmid (3l+2)$, we must have that $2^{2t+1} \mid p$ and $v=3l+2$.
    \item From (1), $p=T(u)=2\cdot 4^t(3l+2)$. Since $l$ is even, $2^e \mid (3l+2)$, and $2^{e+1} \nmid (3l+2)$, we have $p=T(u)=2\cdot4^t\cdot2^e\left( \frac{3l+2}{2^e}  \right)=2^{2t+e+1}\left( \frac{3l+2}{2^e}  \right)$. Thus, $2^{2t+e+1} \mid p$ and $v=\frac{3l+2}{2^e}$.
\end{enumerate}
\end{proof}

\begin{example}
    Let $t=1$ in Theorem 2.28.1. Then, $u=4k+\frac{4-1}{3}=4k+1$. Suppose $k=3$ and $l=1$. Then, $u=13$ and $p=T(u)=40$. Here, we have $2^{2t+1}=8 \mid p$ and $v=3l+2=5$.
\end{example}

\begin{example}
    Let $t=2$ in Theorem 2.28.1. Then, $u=16k+\frac{4^2-1}{3}=16k+5$. Suppose that we still have $k=3$ and $l=1$. Then, $u=53$ and $p=T(u)=160$. Here, we have $2^{2t+1}=32 \mid p$ and $v=3l+2=5$. 
\end{example}

\begin{example}
    Let $t=1$ in Theorem 2.28.2. Then, $u=4k+1$. Let $l=4$ and $k=9$. Then, $u=37$ and $p=T(u)=112$. Now, $2^e|(3l+2) \implies 2^e \mid 14$, so $e=1$. Hence, $2^{2t+e+1}=16$ and $16 \mid 112$. Lastly, $v=\frac{3l+2}{2^e}=7$.
\end{example}

\begin{example}
    Let $t=2$ in Theorem 2.28.2. Then, $u=16k+5$. Let $l=4$ and $k=9$. Then, $u=149$ and $p=T(u)=448$. Now, $2^e|(3l+2) \implies 2^e \mid 14$, so once again $e=1$. Hence, $2^{2t+e+1}=64$ and $64 \mid 448$. Lastly, $v=\frac{3l+2}{2^e}=7$.
\end{example}

While these theorems and examples are valuable insights into Collatz dynamics, there are certainly more relationships to uncover. Further peak value analysis can be executed using algorithm 3 in section 3.

In an effort to make progress in the proof of the Collatz conjecture, relationships such as those previously described should be explored. These relationships decrease the amount of positive integers that need to be checked against the Collatz conjecture. If more relationships were to be discovered that cover the entire space of $\mathbb{N}$, then the Collatz conjecture would be proven. To aid the exploration of such relationships, three novel algorithms are provided in the following section.

\section{The Algorithms}\label{results}
The provided theorems and corollaries serve as a great introduction to a family of mathematical relationships that may cover all of $\mathbb{N}$. There are undoubtedly more relationships to discover pertaining to the Collatz conjecture. It is our hope that the three algorithms that follow will aid in the discovery of such relationships and, consequently, the proof of the conjecture.

\subsection{Visualizing the Collatz Conjecture}

Most attempts to visualize the Collatz conjecture come in the form of inverse binary trees. For example, in Fig. 1, the \textit{left} parent of a node is the even integer that results in the child node, and the \textit{right} parent of a node, if it exists, is the odd integer that results in the child node. Some nodes, such as ``8", only have a left parent node: 16. Note that there does not exist an odd integer $n$ such that $3n+1=8$; therefore, ``8" does not have a right parent. However, ``16" has both a left and a right parent node because $T(32)=16$ and $T(5)=16$.

Algorithm 1 provides a framework for building a visualization tool, where one can construct a custom, interactive interface for navigating an inverse binary tree similar to that in Fig. 1. Such an algorithm can facilitate efficient exploration of the properties of the Collatz conjecture. 

Console output formatting is subject to the reader's taste. Note that the for-loops cycle through 10 values. This number of cycles was arbitrarily chosen and can be changed without affecting the accuracy of the algorithm. The code used to generate Fig. 1 is available as discussed in the \textit{code availability} section at the end of the manuscript.

\begin{algorithm}
	\caption{Inverse Binary Tree Generation}
	\begin{algorithmic}[1]
		
		%\Procedure{Roy}{$a,b$}       \Comment{This is a test}
		\STATE Obtain integer $n$ for the current child node
		\WHILE{input does not fail}
		
		\FOR{(int $i=0; i < 9; i=i+1$)}
		\STATE Output $n*pow(2,i+1)$
		\ENDFOR
		
		\IF{right parent exists for $n$}
		\STATE Output $\frac{n-1}{3}$
		\FOR{(int $i=0; i < 9; i=i+1$)}
		\STATE Output $\frac{n-1}{3}*pow(2,i+1)$
		\IF{right parent exists for $m=\frac{n-1}{3}$}
		\STATE Output $\frac{m-1}{3}$
		\ENDIF
		\ENDFOR
		\ENDIF
		
		\STATE Enter a new value $n$ for a child node
		
		\ENDWHILE
		
	\end{algorithmic}
\end{algorithm}

\subsection{Analyzing Iterations of the Collatz Conjecture}

As stated in the previous section, a rigorous study of relationships regarding the Collatz conjecture will aid in this conjecture's proof. While much attention has been placed on proving the Collatz conjecture for subsets of $\mathbb{N}$, we propose an alternative route: analyzing the number of iterations it takes for some positive integer $n$ to either (1) reach 1 or (2) reach a value less than itself ($<n$).

Algorithm 2 generates output containing the number of iterations each value $n$ requires to reach 1, the number of iterations until it reaches a value less than $n$, and the representation of $n$ in bases 2-4. An analysis of this output for various subsets of $\mathbb{N}$ may result in novel discoveries that can ignite new lines of attack to prove this conjecture. For example, some positive integers such as 381,727 and 2,788,008,987 require 282 and 729 iterations, respectively, to reach a value lower than itself. Perhaps an analysis of the properties of these numbers in various bases will render fruitful discussion. An example of output from this algorithm is shared in Table 4.

\begin{algorithm}
	\caption{Iteration Analysis}
	\begin{algorithmic}[1]
		
		%\Procedure{Roy}{$a,b$}       \Comment{This is a test}
		\STATE Obtain the maximum number for the algorithm to consider: $M$
		
		\FOR{int $i=3;i\le n;i=i+2$}
		
		\STATE Set counters $c_0=0$ and $c_1=0$, and set $i_0=i$
		\WHILE{$i \ne 1$}
		
		\IF{$i\%2==0$}
		\STATE Set $i=i/2$
		\ELSE[$i\%2 \ne 0$]
		\STATE Set $i=3*i+1$
		\ENDIF
		
		\STATE Increase $c_0=c_0+1$
		\IF{$i< i_0$ \&\& $c_1==0$}
		\STATE Set $c_1=c_0$
		\ENDIF
		
		\ENDWHILE
		
		\STATE Output $i$, $c_0$, $c_1$
		\STATE Convert $i$ to bases 2, 3, 4, then output
		
		\ENDFOR
		
	\end{algorithmic}
\end{algorithm}

This algorithm can be easily modified to obtain the maximum number of iterations required to reach 1 for all natural numbers up to $M$, which we refer to as the starting number since it is specified at the beginning of the algorithm. The output can then be analyzed to determine how the number of iterations increases as $M$ increases. Fig. 4 denotes this relationship up to $M=3$ billion, which appears logistic in nature; this provides hope that the conjecture is true, since it appears that the number of iterations converges to a finite number as $M$ increases.

% \usepackage{booktabs}
\begin{table}[]
	\centering
	\caption{Output from Algorithm 2}
	\begin{tabular}{@{}cccccc@{}}
		\toprule
		\textbf{N}  & \textbf{Iterations Until 1} & \textbf{Iterations Until $<$ N} & \textbf{Base 2} & \textbf{Base 3} & \textbf{Base 4} \\ \midrule
		\textbf{3}  & 7                           & 6                             & 11              & 10              & 3               \\
		\textbf{5}  & 5                           & 3                             & 101             & 12              & 11              \\
		\textbf{7}  & 16                          & 11                            & 111             & 21              & 13              \\
		\textbf{9}  & 19                          & 3                             & 1001            & 100             & 21              \\
		\textbf{11} & 14                          & 8                             & 1011            & 102             & 23              \\
		\textbf{13} & 9                           & 3                             & 1101            & 111             & 31              \\
		\textbf{15} & 17                          & 11                            & 1111            & 120             & 33              \\
		\textbf{17} & 12                          & 3                             & 10001           & 122             & 101             \\
		\textbf{19} & 20                          & 6                             & 10011           & 201             & 103             \\
		\textbf{21} & 7                           & 3                             & 10101           & 210             & 111             \\ \bottomrule
	\end{tabular}
\end{table}

\begin{figure}
	\caption{Alternative Output from Algorithm 2}
	\centering
	\includegraphics[scale=.8]{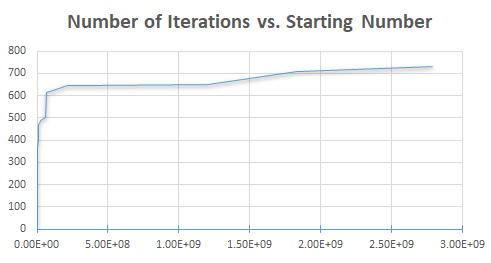}
\end{figure}

\subsection{Determining Peaks in Collatz Sequences}

It may be of interest to explore the highest value that a given natural number $n$ will reach before completing its Collatz sequence. For example, the sequence starting with $n=15$ has a maximum value (or, peak) of 160 before it reaches 1 on its 17th iteration. Perhaps an analysis of these peak numbers will reveal insights related to Collatz sequences and the conjecture. Algorithm 3 determines the peak values of the Collatz sequences starting at each odd natural number up to a given $N$, as well as the number of iterations it takes for each number to reach 1. An example of similar output can be found in tables in \cite{dcode}.

\begin{algorithm}
	\caption{Peak Analysis}
	\begin{algorithmic}[1]
		
		\STATE $\text{Obtain the maximum number for the algorithm to consider: }N$
		
		\FOR{$\text{int }i=3;i\le N;i=i+2$}
		
		\STATE $\text{Set tmp}=i\text{ and peak}=i,\text{ and set counter }c=0$
		\WHILE{$i \ne 1$}
		
		\IF{$\text{tmp}\%2==0$}
		\STATE $\text{Set tmp}=\text{tmp}/2$
		\ELSE[$\text{tmp}\%2 \ne 0$]
		\STATE $\text{Set tmp}=3*\text{tmp}+1$
		\ENDIF
		
		\IF{$\text{peak}<= \text{tmp}$}
		\STATE $\text{Set peak}=\text{tmp}$
		\ENDIF
		
		\STATE $c=c+1$
		
		\ENDWHILE
		
		\STATE $\text{Output }i, \text{peak}, c$
		
		\ENDFOR
		
	\end{algorithmic}
\end{algorithm}

%In this algorithm, we restrict our attention to positive, odd integers. However, the for-loop can be modified on line 2 to include all natural numbers. One observation that intrigued us was that roughly one-third of the positive odd-integers reach a peak value of 9,232 before satisfying the Collatz conjecture.

In this algorithm, we restrict ourselves to positive odd integers up to 100,000. However, one can modify the for-loop in line 2 to include as large a natural number as necessary. One intriguing observation is that of the first 1539 odd natural numbers, for 407 (26.44\%) of them, the peaks occur at 9232. No other peak number has this high of a percentage, with the second most frequent peak in the 1\% range. It is not obvious if 9232 has any significance or if the peaks are somehow related. 

\section{Conclusion}

Although the proof of the Collatz conjecture seems intractable at this time, the contributions of various scholars from around the world are forming an excellent foundation for future approaches. Relationships that cover subsets of $\mathbb{N}$ decrease the amount of natural numbers that need to be checked against the conjecture. If enough relationships are provided, or if a generic relationship can be established, the proof of the Collatz conjecture will be within reach.

The novel theorems, corollaries, and algorithms presented in this manuscript reveal insightful relationships related to the Collatz conjecture. While we present a few relationships between subsets of $\mathbb{N}$, their peak values, and the conjecture, the delineated work should assist in the discovery of more relationships. Additionally, the algorithms presented here may serve as excellent tools to explore various properties between $\mathbb{N}$ and the Collatz conjecture.

\section{Code Availability}

All code used for this manuscript is publicly available at: 

\noindent https://github.com/michaelschwob/CollatzConjecture.

\section{Conflicts of Interest}

None of the authors have conflicting interests.

\section{Acknowledgments}

The authors would like to thank professor Curtis Cooper for providing references \cite{tao2019} and \cite{hartnett2019}, professor Aruhn Venkat for computational assistance, and the anonymous referees for their assistance in this publication. This article's preprint can be found in \cite{schwob2021preprint}.
	
%\bibliographystyle{newjocaa}
%\bibliography{paper}

\printbibliography
\end{document}